\documentclass[12pt,leqno]{article}

\usepackage{amsmath,amsthm,amsfonts,latexsym,amscd,amssymb,hyperref}
\numberwithin{equation}{section}
\input xypic
\def\qed{{\hbadness=10000\hfill\ \vbox{\hrule height.09ex
   \hbox{\vrule width.09ex height1.55ex depth.2ex \kern1.8ex
   \vrule width.09ex height1.55ex depth.2ex}\hrule height.09ex}\break
   \bigskip}}
\newtheorem{theorem}{Theorem}[section]

\newtheorem{corollary}{Corollary}[section]

\theoremstyle{definition}

\theoremstyle{remark}

\begin{document}

\linespread{1}\title{\textbf{Killing Correspondence in Finsler spaces}}

\author{\\\textbf{Suresh K. Shukla and M. K. Gupta}\\
\normalsize{Department of Pure and Applied Mathematics}\\ \normalsize{Guru Ghasidas Vishwavidyalaya}\\ \normalsize{Bilaspur (C.G.)-495009, INDIA}\\
\normalsize{Email: shuklasureshk@gmail.com; mkgiaps@gmail.com}}
\date{}
\maketitle

\linespread{1.3}\begin{abstract} The present paper deals with the Killing correspondence between some Finsler spaces. We consider a Finsler space equipped with a $\beta$-change of metric and study the Killing correspondence between the original Finsler space and the Finsler space equipped with $\beta$-change of metric. We obtain necessary and sufficient condition for a vector field Killing in the original Finsler space to be Killing in the Finsler space equipped with $\beta$-change of metric. Certain consequences of such result are also discussed.

\textbf{Keywords:} Finsler spaces, $\beta$-change, Killing vector field \

\textbf{2010 Mathematics Subject Classification:} 53B40.
\end{abstract}

\section{Introduction}
~~~~As a matter of investigation, it is important to observe how properties of a Finsler space change under a change in the metric. Several geometers from different parts of the globe have  been working in this direction for the last 2-3 decades. M.S. Knebelman \cite{Kn}, S.Golab \cite{Go1} and M.Hashiguchi \cite{Ha} studied conformal change of Finsler metrics. Park and Lee \cite{Le} discussed various Randers changes of Finsler spaces with $(\alpha,\beta)$-metrics of Douglas type. M. Matsumoto \cite{Ma1} and T. Aikou \cite{Ai} studied and investigated several properties of projective change and projective Randers change. In 1984, C. Shibata \cite{Shi} studied $\beta$-change of Finsler metrics and discussed certain invariant tensors under such a change.

Killing equations play important role in the study of a Finsler space which undergoes a change in the metric. In fact, they give an equivalent characterization for the transformations to preserve distances. In 1979, Singh, \emph{et. al.} \cite{Si} studied a Randers space $F^n \left(M, L(x,y)=\left(g_{i\,j} (x)\,y^i\,y^j \right)^{\frac{1}{2}}+b_i (x)\,y^i \right)$,\;$n\geq 2$ which undergoes a change $L(x,y) \mapsto L^{*} (x,y)=L^2 (x,y) + \left(\alpha_i (x) y^i\right)^2$. They discussed Killing correspondence of the spaces $F^n (M,L)$ and $F^{\,*\,n} (M,L^*)$.

In the present paper, we consider a general Finsler space $F^n (M,L)$ which undergoes a $\beta$-change, that is $L(x,y) \mapsto \bar{L} (x,y)=f(L,\,\beta)$, where $\beta (x,y)=b_i (x) y^i$ is a 1-form. We study Killing correspondence of the Finsler spaces $F^n (M,L)$ and $\bar{F}^{\,n} (M,\bar{L})$. For the notations and terminology, we refer the reader to the books \cite{Pl01} and \cite{Ru}, and the paper \cite{Shi} by Shibata.

The paper is organized as follows. In section 2, we give some preliminaries which are used in the discussion of subsequent sections. Section 3 deals with Killing correspondence of $F^n (M,L)$ and $\bar{F}^{\,n} (M,\bar{L})$, where $\bar{L} (x,y)=f(L,\,\beta)$. In section 4, we give conclusion to the results obtained in the paper and discuss future possible work to be done in this direction.

\section{Preliminaries} 

~~~~Let $F^n (M,L)$, $n \geq 2$ be an $n$-dimensional  Finsler space. Suppose that the metric function $L(x,y)$ undergoes a change $L(x,y) \mapsto \bar{L} (x,y)=f(L,\,\beta)$, where $\beta (x,y)=b_i (x) y^i$ is a 1-form and the new space is $\bar{F}^n (M,\bar{L})$. This change of metric is called a $\beta$-change (\emph{see} \cite{Shi} and \cite{Pl01}). 

The angular metric tensor $\bar{h}_{ij}$ of the space $\bar{F}^{\,n}$ is given by \cite{Shi} 
\begin{equation}\label{a}
\bar{h}_{ij}=p h_{ij}+q_0 m_i m_j,
\end{equation}
where
\begin{equation}\left\{ \begin{split}
& p=f\,f_1/L,\;\;q_0=f\,f_{22},\;\;m_i=b_i-\beta y^i/L^2,\\ 
& f_1=\partial f/\partial L,\;\;f_2=\partial f/\partial \beta,\\ 
\end{split}\right. \end{equation}
$h_{ij}$ being the angular metric tensor of $F^n$.
The fundamental metric tensor $\bar{g}_{ij}$ and its inverse $\bar{g}^{ij}$ of $\bar{F}^{\,n}$ are expressed as \cite{Shi}
\begin{equation}
\bar{g}_{ij}=p g_{ij}+p_0 b_i b_j+ p_{-1} (b_i y_j+b_j y_i)+p_{-2} y_i y_j,
\end{equation}
\begin{equation}
\bar{g}^{ij}=g^{ij}/p -s b^i b^j-s_{-1} (b^i y^j+b^j y^i)-s_{-2} y^i y^j,
\end{equation}
where
\begin{equation}\left\{ \begin{split}
& p_0=q_0+f_2^2,\\ 
& q_{-1}=f\,f_{12}/L,\;\;p_{-1}=q_{-1}+p\,f_2/f,\;\;q_{-2}=f\left(f_{11}-f_1/L\right)/L^2,\\
& p_{-2}=q_{-2}+p^2/f^2,\\ 
& b^i=g^{ij} b_j,\;\;b^2=g^{ij} b_i b_j,\;\;s_0=\bar{L} q_0/(\tau p L^2),\\
& s_{1}=p_{-1} \bar{L}^2/(\tau p L^2),\;\;s_{-2}=p_{-1} (\nu p L^2-b^2 \bar{L}^2)/(\tau p L^2 \beta),\\
& \tau=\bar{L}^2 (p+\nu q_0)/L^2,\;\;\nu=b^2-\beta^2/L^2,
\end{split}\right. \end{equation}
$g_{ij}$ and $g^{ij}$ respectively being the metric tensor and inverse metric tensor of $F^n$. The Cartan tensor $\bar{C}_{ijk}$ and the associate Cartan tensor $\bar{C}^h_{ij}$ of $\bar{F}^{\,n}$ are given by the following expressions:
\begin{equation}\label{b}
\bar{C}_{ijk}=p\,C_{ijk}+\frac{1}{2} p_{-1} \mathfrak{S}_{(ijk)} \{h_{ij}\,m_k\}+\frac{1}{2} p_{02}\,m_i m_j m_k, 
\end{equation}
 \begin{equation}\label{c}
\bar{C}^h_{ij}=C^h_{ij}-V^h_{ij},
\end{equation}
where
\begin{equation}\begin{split}
V^h_{ij}=&Q^h (p C_{imj} b^m-p_{-1} m_i m_j)-(\frac{1}{p}m^h-\nu Q^h) (p_{02} m_i m_j+p_{-2} h_{ij})/2 \,\\
&-p_{-1} (h^h_i m_j + h^h_j m_i)/(2p),
\end{split}
\end{equation}

\begin{equation}
Q^h=s_0 b^h +s_{-1} y^h,\;\;h^h_i=g^{hr} h_{ir},\;\;m^h=g^{hr} m_r,\;\;p_{02}=\partial p_0/\partial \beta,
\end{equation}             
$\mathfrak{S}_{(ijk)}$ denote the cyclic sum with respect to the indices \textit{i}, \textit{j} and \textit{k}; $C_{ijk}$ and $C^h_{ij}$ respectively being the Cartan tensor and associate Cartan tensor of $F^n$. 

The spray coefficients $\bar{G}^i$ of $\bar{F}^{\,n}$ in terms of the spray coefficients $G^i$ of $F^n$ are expressed as \cite{Shi}
\begin{equation}\label{d}
\bar{G}^i =G^i + D^i,
\end{equation}  
where
\begin{equation*}
\begin{split}
& D^i=(q/p)\,F^i_0+(p\,E_{00}-2q F_{r0} b^r) (s_{-1}\,y^i+s_0 b^i)/2,\\
& F^i_j=g^{ir} F_{rj},\;\;E_{jk}=(1/2) (b_{j|k}+b_{k|j}),\;\;F_{jk}=(1/2) (b_{j|k}-b_{k|j}),
\end{split}
\end{equation*}
the symbol '$_|$' denote the $h$-covariant derivative with respect to the Cartan connection $C\Gamma$ and the lower index '$_0$' (except in $s_0$) denote the contraction by $y^i$.

The relation between the coefficients $\bar{N}^i_j$ of Cartan nonlinear connection in $\bar{F}^{\,n}$ and the coefficients $N^i_j$ of the corresponding Cartan nonlinear connection in $F^n$ is given by \cite{Shi}
\begin{equation}\label{e}
\bar{N}^i_j=N^i_j + D^i_j,
\end{equation} 
where
\begin{equation}
D^i_j=\dot\partial_j D^i,\;\;\dot\partial_j \equiv \partial/\partial y^j.
\end{equation} 
The coefficients $\bar{F}^i_{jk}$ of Cartan connection $C\bar{\Gamma}$ in $\bar{F}^{\,n}$ and the coefficients $F^i_{jk}$ of the corresponding Cartan connection $C\Gamma$ in $F^n$ are related as \cite{Shi}
\begin{equation}\label{f}
\bar{F}^i_{jk}=F^i_{jk}+D^i_{jk},
\end{equation}
where
\begin{equation*}
\begin{split}
D^i_{jk}=& \{(1/p)g^{is}-Q^i b^s-y^s (s_{-1} b^i+s_{-2} y^i)\}\,\\
&(B_{sj} b_{0|k}+B_{sk} b_{0|j}-B_{kj} b_{0|s}+F_{sj} Q_k+F_{sk} Q_j+E_{kj} Q_s+p C_{jkr}D^r_s\,\\
&+V_{jkr}D^r_s-p C_{skm}D^m_j-V_{sjm}D^m_k-p C_{sjm}D^m_k-V_{skm}D^m_j);
\end{split}
\end{equation*}

\begin{equation*}
B_{kj}=2 \dot\partial_j Q_k.
\end{equation*}

The difference tensor $D^i_{jk}$ satisfies the following properties:
\begin{equation*}
(i)\;\;\;D^i_{j0}=B^i_{j0}=D^i_j,\;\;\;(ii)\;\;\;D^i_{00}=2D^i,\;\;\text{where}\;\;B^i_{jk}=\dot\partial_k D^i_j.
\end{equation*}

\section{Killing Correspondence of $F^n$ and $\bar{F}^n$}
~~~Let us consider an infinitesimal transformation 
\begin{equation}\label{g}
^\prime{x}^i=x^i+\epsilon v^i(x),
\end{equation}
where $\epsilon$ is an infinitesimal constant and $v^i(x)$ is a contravariant vector field.

The vector field $v^i(x)$ is said to be a Killing vector field in $F^n$ if the metric tensor of the Finsler space with respect to the infinitesimal transformation (\ref{g}) is Lie invariant, that is 
\begin{equation}\label{t}
\mathfrak{L}_v g_{ij}=0,
\end{equation}
 $\mathfrak{L}_v$ being the operator of Lie differentiation. Equivalently, the vector field $v^i(x)$ is Killing in $F^n$ if 
\begin{equation}\label{h}
v_{i|j}+v_{j|i}+2 C^h_{ij} v_{h|0}=0,
\end{equation}
where $v_i=g_{il} v^l$.

Now, we prove the following result which gives a necessary and sufficient condition for a Killing vector field in $F^n$ to be Killing in $\bar{F}^n$:
\begin{theorem}\label{A}
A Killing vector field $v^i(x)$ in $F^n$ is Killing in $\bar{F}^n$ if and only if
\begin{equation}\label{i}
V^h_{ij}\,v_{h|0}+C_{r\,j\,l}\,v^l\,D^r_i+C_{r\,i\,l}\,v^l\,D^r_j+v_r\,D^r_{ij}+\bar{C}^h_{ij}\,\,\left(2C_{r\,h\,l}\,v^l\,D^r+v_r\,D^r_h\right)=0,
\end{equation} 
where $\bar{C}^h_{ij}$ is the associate Cartan tensor of $\bar{F}^n$.
\end{theorem} 
\begin{proof}
Assume that $v^i(x)$ is Killing in $F^n$. Then (\ref{h}) is satisfied. By definition, the $h$-covariant derivatives of $v_i$ with respect to $C\bar{\Gamma}$ and $C\Gamma$ are respectively given as
\begin{equation}\label{j}
(a)\;\;v_{i||j}=\partial_j v_i-(\dot\partial_r v_i) \bar{G}^r_j-v_r \bar{F}^r_{ij},\;\;\;(b)\;\;\;\;v_{i|j}=\partial_j v_i-(\dot\partial_r v_i) G^r_j-v_r F^r_{ij},
\end{equation}
where $\partial_j \equiv \partial/\partial x^j$ and '$_{||}$' denote the $h$-covariant differentiation with respect to $C\bar{\Gamma}$. Equation (\ref{j})(a), by virtue of (\ref{d}), (\ref{f}) and (\ref{j})(b), takes the form
\begin{equation}\label{k}
v_{i||j}=v_{i|j}-2\,C_{r\,i\,l} v^l D^r_j-v_r D^r_{ij}.
\end{equation}  
Now, from (\ref{k}), we have 
\begin{equation}\label{l}
\begin{split}
v_{i||j}+v_{j||i}+2 \bar{C}^h_{ij} v_{h||0}=& v_{i|j}+v_{j|i}+2 \bar{C}^h_{ij} v_{h|0}-2\,C_{r\,i\,l} v^l D^r_j-2\,C_{r\,j\,l} v^l D^r_i\,\\
& -2 v_r D^r_{ij}-2\,\bar{C}^h_{ij} (2 C_{rhl} v^l D^r+v_r D^r_h).
\end{split}
\end{equation}
Using (\ref{c}) in (\ref{l}) and applying (\ref{h}), we get
\begin{equation}\label{m}
\begin{split}
v_{i||j}+v_{j||i}+2 \bar{C}^h_{ij} v_{h||0}=& -2 V^h_{ij} v_{h|0}-2\,C_{r\,i\,l} v^l D^r_j-2\,C_{r\,j\,l} v^l D^r_i\,\\
& -2 v_r D^r_{ij}-2\,\bar{C}^h_{ij} (2 C_{rhl} v^l D^r+v_r D^r_h).
\end{split}
\end{equation}
Proof completes with the observation that $v^i(x)$ is Killing in $\bar{F}^n$ if and only if $v_{i||j}+v_{j||i}+2 \bar{C}^h_{ij} v_{h||0}=0$, that is, if and only if (\ref{i}) holds.
\end{proof}

If a vector field $v^i(x)$ is Killing in $F^n$ and $\bar{F}^n$, then from Theorem \ref{A}, (\ref{i}) holds, which on transvection by $y^i$ yields 
\begin{equation}\label{n}
2 C_{rjl} v^l D^r+v_r D^r_j =0.
\end{equation}
Equation (\ref{i}), in view of (\ref{n}), enables us to state the following:
\begin{corollary}\label{B}
If a vector field $v^i(x)$ is Killing in $F^n$ and $\bar{F}^n$, then 
\begin{equation}\label{o}
V^h_{ij}\,v_{h|0}+C_{r\,j\,l}\,v^l\,D^r_i+C_{r\,i\,l}\,v^l\,D^r_j+v_r\,D^r_{ij}=0.
\end{equation}
\end{corollary} 

As another important consequence of Theorem \ref{A}, we have the following:
\begin{corollary}\label{C}
If a vector field $v^i(x)$ is Killing in $F^n$ and $\bar{F}^n$, then the vector $v_i(x,y)$ is orthogonal to the vector $D^i(x,y)$.
\end{corollary}
\begin{proof}
As $v^i(x)$ is Killing in $F^n$ and $\bar{F}^n$, (\ref{i}) holds, which on transvection by $y^i$ gives (\ref{n}). Again transvecting (\ref{n}) by $y^j$, it follows that $v_r\,D^r=0$. This proves the result. 
\end{proof}     

\section{Discussion and Conclusion}
~~~We proved Theorem \ref{A} as the main result and as its consequences we obtained Corollary \ref{B} and Corollary \ref{C}. Since the Killing equation (\ref{t}) is a necessary and sufficient condition for the transformation (\ref{g}) to be a motion in $F^n$ (\emph{vide} \cite{Ru}), the condition (\ref{i}) obtained in Theorem \ref{A} may be taken as the necessary and sufficient condition for the  vector field $v^i(x)$, generating a motion in $F^n$, to generate a motion in $\bar{F}^n$ as well. Also, since every motion is an affine motion and every affine motion is a projective motion (\emph{vide} \cite{Pa1}-\cite{Pa3}), it is clear that vector field $v^i(x)$, generating an affine motion (respectively projective motion) in $F^n$, generates an affine motion (respectively projective motion) in $\bar{F}^n$ if condition (\ref{i}) holds. The main result and its consequences, obtained in the paper, may be further utilized to link various transformations in $F^n$ with corresponding transformations in $\bar{F}^n$.



\begin{thebibliography}{}

\bibitem[1]{Kn} M. S. Knebelman, Conformal geometry of generalized metric spaces, \textit{Proc. Nat. Sci., U. S. A.}, \textbf{15} (1929), 376-379.  

\bibitem[2]{Go1} S. Golab, Einige Bemerkungen uber Winkelmetrik in Finslerschen Raumen, \textit{Verh. Intern. Math. Kongr.}, Zurich, \textbf{II} (1932), 178-179.

\bibitem[3]{Ha} M. Hashiguchi, On conformal transformations of Finsler metrics, \textit{J. Math. Kyoto Univ.}, \textbf{16} (1976), 25-50.

\bibitem[4]{Le} H. S. Park and I.Y. Lee, The Randers changes of Finsler spaces with $(\alpha,\beta)$-metrics of Douglas type, \textit{J. Korean Math. Soc.}, \textbf{8(3)} (2001), 503-521.

\bibitem[5]{Ai} T. Aikou, Some remarks on conformal changes of generalized Finsler metrics, \textit{Rep. Fac. Sci., Kagoshima Univ., (Math., Phys. \& Chem.)}, \textbf{20} (1987), 57-61. 

\bibitem[6]{Ma1} M. Matsumoto, The Tavakol-van Den Berg conditions in the theories of gravity and projective changes of Finsler metrics, \textit{Publ. Math. Debrecen}, \textbf{42} (1993), 155-168.

\bibitem[7]{Si} U.P. Singh, V. N. John and B. N. Prasad, Finsler spaces preserving Killing vector fields, \textit{J. Math. Phys. Sci. Madras}, \textbf{13} (1979), 265-271. 

\bibitem[8]{Shi} C. Shibata, On invariant tensors of $\beta$-changes of Finsler metrics, \textit{J. Math. Kyoto Univ.}, \textbf{24} (1984), 163–188.

\bibitem[9]{Pl01} P. L. Antonelli (ed.), \textit{Handbook of Finsler geometry}, Kluwer Acad. Publ., Dordrecht, 2003.

\bibitem[10]{Ru} H. Rund, \textit{The Differential Geometry of Finsler Spaces}, Springer-Verlag, 1959.

\bibitem[11]{Pa1} P. N. Pandey, Certain types of affine motion in a Finsler manifold I, II, III. \textit{Colloq. Math.}, \textbf{49} (1985), 243-252; \textbf{53} (1987), 219-227; \textbf{56} (1988), 333-340.

\bibitem[12]{Pa2} P. N. Pandey, Certain types of projective motion in a Finsler manifold, \textit{Atti. Accad. Peloritana Pericolanti Cl. Sci. Fis. Mat. Natur.}, \textbf{60} (1983) 287-300.

\bibitem[13]{Pa3} P. N. Pandey, Certain types of projective motion in a Finsler manifold II, A\textit{tti. Accad. Sci. Torino}, \textbf{120 (5-6)} (1986), 168-178.


\end{thebibliography}
\end{document}